\documentclass[12pt]{amsart}
                                               \usepackage{amsmath,amssymb,amsbsy,amsfonts,amsthm,latexsym,
            amsopn,amstext,amsxtra,euscript,amscd,bm}
                                               
\usepackage{url}
\usepackage[colorlinks,linkcolor=blue,anchorcolor=blue,citecolor=blue,backref=page]{hyperref}
\usepackage{color}
\usepackage{graphics,epsfig}
\usepackage{graphicx}
\usepackage{float}

\hypersetup{breaklinks=true}

\renewcommand*{\backref}[1]{}
\renewcommand*{\backrefalt}[4]{%
    \ifcase #1 (Not cited.)%
    \or        (p.\,#2)%
    \else      (pp.\,#2)%
    \fi}

\begin{document}

\newtheorem{theorem}{Theorem}
\newtheorem{lemma}[theorem]{Lemma}
\newtheorem{claim}[theorem]{Claim}
\newtheorem{cor}[theorem]{Corollary}
\newtheorem{prop}[theorem]{Proposition}
\newtheorem{definition}{Definition}
\newtheorem{question}[theorem]{Question}
\newcommand{\hh}{{{\mathrm h}}}

\numberwithin{equation}{section}
\numberwithin{theorem}{section}
\numberwithin{table}{section}

\def\sssum{\mathop{\sum\!\sum\!\sum}}
\def\ssum{\mathop{\sum\ldots \sum}}
\def\dsum{\mathop{\sum\sum}}
\def\iint{\mathop{\int\ldots \int}}

\def\squareforqed{\hbox{\rlap{$\sqcap$}$\sqcup$}}
\def\qed{\ifmmode\squareforqed\else{\unskip\nobreak\hfil
\penalty50\hskip1em\nobreak\hfil\squareforqed
\parfillskip=0pt\finalhyphendemerits=0\endgraf}\fi}

\newfont{\teneufm}{eufm10}
\newfont{\seveneufm}{eufm7}
\newfont{\fiveeufm}{eufm5}
%
%
\newfam\eufmfam
     \textfont\eufmfam=\teneufm
\scriptfont\eufmfam=\seveneufm
     \scriptscriptfont\eufmfam=\fiveeufm
%
%
\def\frak#1{{\fam\eufmfam\relax#1}}

\newcommand{\bflambda}{{\boldsymbol{\lambda}}}
\newcommand{\bfmu}{{\boldsymbol{\mu}}}
\newcommand{\bfxi}{{\boldsymbol{\xi}}}
\newcommand{\bfrho}{{\boldsymbol{\rho}}}

\def\fK{\mathfrak K}
\def\fT{\mathfrak{T}}

\def\fA{{\mathfrak A}}
\def\fB{{\mathfrak B}}
\def\fC{{\mathfrak C}}
\def\fM{{\mathfrak M}}

\def \balpha{\bm{\alpha}}
\def \bbeta{\bm{\beta}}
\def \bgamma{\bm{\gamma}}
\def \blambda{\bm{\lambda}}
\def \bchi{\bm{\chi}}
\def \bphi{\bm{\varphi}}
\def \bpsi{\bm{\psi}}
\def \bomega{\bm{\omega}}
\def \btheta{\bm{\vartheta}}

\def \bxi{\bm{\xi}}

\def\eqref#1{(\ref{#1})}

\def\vec#1{\mathbf{#1}}


\def\cA{{\mathcal A}}
\def\cB{{\mathcal B}}
\def\cC{{\mathcal C}}
\def\cD{{\mathcal D}}
\def\cE{{\mathcal E}}
\def\cF{{\mathcal F}}
\def\cG{{\mathcal G}}
\def\cH{{\mathcal H}}
\def\cI{{\mathcal I}}
\def\cJ{{\mathcal J}}
\def\cK{{\mathcal K}}
\def\cL{{\mathcal L}}
\def\cM{{\mathcal M}}
\def\cN{{\mathcal N}}
\def\cO{{\mathcal O}}
\def\cP{{\mathcal P}}
\def\cQ{{\mathcal Q}}
\def\cR{{\mathcal R}}
\def\cS{{\mathcal S}}
\def\cT{{\mathcal T}}
\def\cU{{\mathcal U}}
\def\cV{{\mathcal V}}
\def\cW{{\mathcal W}}
\def\cX{{\mathcal X}}
\def\cY{{\mathcal Y}}
\def\cZ{{\mathcal Z}}
\newcommand{\rmod}[1]{\: \mbox{mod} \: #1}

\def\cg{{\mathcal g}}

\def\vr{\mathbf r}

\def\e{{\mathbf{\,e}}}
\def\ep{{\mathbf{\,e}}_p}
\def\eq{{\mathbf{\,e}}_q}

\def\Tr{{\mathrm{Tr}}}
\def\Nm{{\mathrm{Nm}}}

 \def\SS{{\mathbf{S}}}

\def\lcm{{\mathrm{lcm}}}

\def\({\left(}
\def\){\right)}
\def\fl#1{\left\lfloor#1\right\rfloor}
\def\rf#1{\left\lceil#1\right\rceil}

\def\mand{\qquad \mbox{and} \qquad}

\newcommand{\commK}[1]{\marginpar{%
\begin{color}{red}
\vskip-\baselineskip 
\raggedright\footnotesize
\itshape\hrule \smallskip B: #1\par\smallskip\hrule\end{color}}}

\newcommand{\commI}[1]{\marginpar{%
\begin{color}{magenta}
\vskip-\baselineskip 
\raggedright\footnotesize
\itshape\hrule \smallskip I: #1\par\smallskip\hrule\end{color}}}

\newcommand{\commT}[1]{\marginpar{%
\begin{color}{blue}
\vskip-\baselineskip 
\raggedright\footnotesize
\itshape\hrule \smallskip I: #1\par\smallskip\hrule\end{color}}}




\hyphenation{re-pub-lished}

\mathsurround=1pt

\def\bfdefault{b}

\def \F{{\mathbb F}}
\def \K{{\mathbb K}}
\def \Z{{\mathbb Z}}
\def \Q{{\mathbb Q}}
\def \R{{\mathbb R}}
\def \C{{\mathbb C}}
\def\Fp{\F_p}
\def \fp{\Fp^*}

\def\Kmnp{\cK_p(m,n)}
\def\Kmnq{\cK_q(m,n)}
\def\Klmnp{\cK_p(\ell, m,n)}
\def\Klmnq{\cK_q(\ell, m,n)}

\def \SALMNq {\cS_q(\balpha;\cL,\cM,\cN)}
\def \SALMNp {\cS_p(\balpha;\cL,\cM,\cN)}
\def \SAqLMNq {\cS_q(\balpha_q;\cL,\cM,\cN)}
\def \SApLMNp {\cS_p(\balpha_p;\cL,\cM,\cN)}

\def\SAMJp{\cS_p(\balpha;\cM,\cJ)}
\def\SAMJq{\cS_q(\balpha;\cM,\cJ)}
\def\SAJq{\cS_q(\balpha;\cJ)}
\def\SAIJp{\cS_p(\balpha;\cI,\cJ)}
\def\SAIJq{\cS_q(\balpha;\cI,\cJ)}

\def\RIJp{\cR_p(\cI,\cJ)}
\def\RIJq{\cR_q(\cI,\cJ)}

\def\TWXJp{\cT_p(\bomega;\cX,\cJ)}
\def\TWXJq{\cT_q(\bomega;\cX,\cJ)}
\def\TWJq{\cT_q(\bomega;\cJ)}

 \def \xbar{\overline x}
  \def \ybar{\overline y}

\title[Trilinear Forms with Double Kloosterman Sums]{Trilinear Forms with Double Kloosterman Sums}

 \author[I. E. Shparlinski] {Igor E. Shparlinski}
 \thanks{This work was  supported   by ARC Grant~DP170100786.}
 
\address{Department of Pure Mathematics, University of New South Wales,
Sydney, NSW 2052, Australia}
\email{igor.shparlinski@unsw.edu.au}

\begin{abstract} We obtain several  estimates for trilinear form with double Kloosterman sums.   
In particular, these bounds show the existence of nontrivial cancellations between such sums. 
\end{abstract}


\maketitle

\section{Introduction}
\label{sec:intro}
\subsection{Background and motivation}

 Let  $q$ be  a positive 
integer. We denote the residue ring modulo $q$ by $\Z_q$ and denote the group
of units of $\Z_q$ by $\Z_q^*$.

 For integers $\ell$, $m$ and $n$ we define
the {\it double Kloosterman sum\/}
$$
\Klmnq = \sum_{x,y \in \Z_q^*} \eq\(\ell xy +m \xbar + n\ybar \).
$$
where $\xbar$ is the multiplicative inverse of $x$ modulo $q$ and
$$
\eq(z) = \exp(2 \pi i z/q).
$$
Given  three sets
 \begin{align*}
\cL & = \{u+1, \ldots, u+L\},  \\ 
\cM & = \{v+1, \ldots,  v+M\}, \\  
\cN & = \{w+1, \ldots,  w+N\},
\end{align*}
of $L$, $M$, $N$ consecutive integers  
and a   sequence of weights $\balpha = \{\alpha_\ell\}_{\ell\in \cL}$,
we define the weighted triple sums of double Kloosterman sums
$$ 
\SALMNq = \sum_{\ell \in \cL} \alpha_\ell \sum_{m\in \cM} \sum_{n \in \cN}\Klmnq.
$$

Assuming that $\alpha_\ell = 0$ if $\gcd(\ell, q)>1$ and using  the Weil bound~\cite[Equation~(11.58)]{IwKow},  one can easily obtain 
$$
\left| \SALMNq \right| \le MN q^{1+o(1)} \sum_{\ell \in \cL}   |\alpha_\ell|, 
$$
which in the case $|\alpha_{\ell}| \le 1$ takes form 
 \begin{equation}
\label{eq:trivial ALMN}
\left| \SALMNq \right| \le LMN q^{1+o(1)}.
\end{equation}

We are interested in studying cancellations amongst Kloosterman sums and
thus in improvements of the trivial bound~\eqref{eq:trivial ALMN}.

This question is partially motivated by a series of recent results concerning 
various bilinear forms with single  Kloosterman sums
$$
\Kmnq = \sum_{x,y \in \Z_q^*} \eq\(m x +n \xbar \), 
$$
see~\cite{BFKMM1,BFKMM2,FKM,KMS,Shp,ShpZha} and references therein for various approaches, 
and  also for generalisation to bilinear forms with more general quantities. 
The triple sums  $\SALMNq $ seems to be a new object of study.

\subsection{Results} 

Here we use some ideas from~\cite{Shp,ShpZha} to improve the trivial bound~\eqref{eq:trivial ALMN}.
Although the approach works in larger generality, to exhibit it in a simplest form we assume that 
weights $\balpha$ supported only on $\ell \in \Z_q^*$, that is, that $\alpha_\ell = 0$ if $\gcd(\ell, q)=1$. 

\begin{theorem}
\label{thm:SALMNq} For any integer $q \ge 1$, and weights 
$\balpha = \{\alpha_\ell\}_{\ell\in \cL}$ with $|\alpha_{\ell}| \le 1$ and
supported only on $\ell \in \Z_q^*$, we have,
 \begin{align*}
 &|\SALMNq|  \\ 
 & \qquad  \le  \min\bigl\{
\(L +    L^{1/2} M^{1/2}\)  N^{1/2} q^{3/2} , \\
 & \qquad  \qquad \qquad\qquad\qquad
 \(L +    L^{3/4} M^{1/4}\) \( N^{1/8}  q^{7/4}  +   N^{1/2} q^{3/2}\) \bigr\}q^{o(1)}. 
\end{align*}
\end{theorem}

\begin{theorem}
\label{thm:SALMNq Aver}
 For any fixed real $\varepsilon > 0$ and integer $r\ge 2$, for any sufficiently large  $Q \ge 1$, for all but  at most $Q^{1-2r \varepsilon + o(1)}$ integers 
$q\in [Q,2Q]$  and weights $\balpha_q=\{\alpha_{q,\ell}\}_{\ell\in \cL}$, that may depend on $q$, 
with $|\alpha_{q,\ell}| \le 1$ and  supported only on $\ell \in \Z_q^*$,  we have,
$$
|\SAqLMNq|   \le  \(L +   L^{1-1/2r} M^{1/2r}\) \(q^{2-1/2r}  +  N^{1/2} q^{3/2}\) q^{o(1)}. 
$$
\end{theorem}

Clearly, the roles of $M$ and $N$ can be interchanged in 
Theorems~\ref{thm:SALMNq} and~\ref{thm:SALMNq Aver}. 

Now, assuming that $N \le q^{2/3+o(1)}$ we have  $N^{1/8}  q^{7/4}  \ge N^{1/2} q^{3/2+o(1)}$.
Hence by Theorems~\ref{thm:SALMNq}  we have
$$
|\SALMNq|   \le   \(L +    L^{3/4} M^{1/4}\) N^{1/8}  q^{7/4+o(1)}, 
$$
which improves~\eqref{eq:trivial ALMN} for 
$$
N \le q^{2/3+o(1)} \mand \min\{L,M\}  M^3 N^{7/2} \ge q^{3+\varepsilon}
$$
for some fixed $\varepsilon>0$. 
Thus in the symmetric case when $L \sim M \sim N$ this condition becomes 
$q^{2/3+o(1)} \ge L \ge q^{2/5+\varepsilon}$.

\subsection{Possible generalisations and open problems}
Analysing the proofs of Theorems~\ref{thm:SALMNq} and~\ref{thm:SALMNq Aver} one can easily see that 
they can be extended to more general sums of the form
$$
 \sum_{x,y \in \Z_q^*}   \eta_{x}  \kappa_y \eq\(\ell xy +m \xbar + n\ybar \), 
$$
with complex weights  satisfying $|\eta_{x}|, |\kappa_y|\le 1$  (which may depend
on $q$ in the settings of Theorem~\ref{thm:SALMNq Aver}). 

On the other hand, our approach does not work for the sums 
$$
 \sum_{\ell \in \cL} \alpha_\ell \sum_{m\in \cM} \beta_m \sum_{n \in \cN}\Klmnq
 \quad \text{and}\quad
 \sum_{\ell \in \cL} \alpha_\ell \sum_{m\in \cM} \beta_m \sum_{n \in \cN} \gamma_n\Klmnq
$$
with nontrivial weights attached to the variables $m$ and $n$. It is however possible that 
one can apply to these sums the method of~\cite{BFKMM1,FKM,KMS}. 

 \section{Preliminaries} 

\subsection{General notation}

We always assume that  the sequence of weights $\balpha = \{\alpha_\ell\}_{\ell \in \cL}$ is 
supported only on $\ell$ with $\gcd(\ell,q)=1$, that is, we have $\alpha_\ell = 0$  if $\gcd(\ell,q)>1$ (and the same for the weights  $\balpha_q=\{\alpha_{q,\ell}\}_{\ell\in \cL}$
 depending on $q$).

Throughout the paper,  as usual $A\ll B$ and $B \gg A$  are equivalent to the inequality $|A|\le cB$
with some  constant $c>0$, which 
occasionally, where obvious, may depend on the real 
parameter $\varepsilon>0$ and on the integer parameter $r \ge 1$, and
 is absolute  otherwise.


  \subsection{Number of solutions to some multiplicative congruences}

We start with some estimates on power moments of character sums.
Let $\cX_q$ be the set of all multiplicative characters $\chi$ modulo $q$ and
let $\cX_q^* = \cX_q \setminus \{\chi_0\}$ be the set of non-principal characters. 

The first result is a special case of a bound of 
Cochrane and Shi~\cite[Theorem~1] {CochShi}.

\begin{lemma}
\label{lem:4th-Moment}
For any integers $k$ and $H$ we have
$$
\sum_{\chi\in\cX_q}
\left|\sum_{z=k+1}^{k+H}\chi(x)\right|^4 \le H^2 q^{o(1)}.
$$
\end{lemma}

We now derive our main technical tool. 

\begin{lemma}
\label{lem:Energy}
For any  sets
$$
\cA = \{s+1, \ldots, s+A\} \mand \cB = \{t+1,\ldots,  t+B\}
$$
consisting of $A$ and $B$ consecutive integers, respectively, for
 \begin{align*}
E(\cA, \cB) = \{a_1 b_1 \equiv  a_2 b_2 \bmod q~:~ 
a_1, a_2 \in \cA , &\  b_1, b_2 \in \cB, \\
& \gcd(a_1a_2b_1b_2,q)=1 \}
\end{align*}
we have
$$
E(\cA, \cB) \le \(\frac{A^2B^2}{q} + AB\) q^{o(1)}.
$$
\end{lemma}

\begin{proof}
Using the orthogonality of characters, we write
$$
E(\cA, \cB) =
\sum_{\substack{a_1, a_2 \in \cA\\
 \gcd(a_1a_2,q)=1}}  \ \sum_{\substack{b_1, b_2 \in \cA\\
 \gcd(b_1b_2,q)=1}}      \frac{1}{\varphi(q)}\sum_{\chi\in \cX_q} \chi\(a_1a_2b_1^{-1}b_2^{-1}\), 
$$
where, as usual, $\varphi(q)$ denotes the 
Euler function. 
Changing the order summation, and separating the contribution 
$$
  \frac{1}{\varphi(q)} \sum_{\substack{a_1, a_2 \in \cA\\
 \gcd(a_1a_2,q)=1}}  \sum_{\substack{b_1, b_2 \in \cB\\
 \gcd(b_1b_2,q)=1}}   1 \le  \frac{A^2B^2}{\varphi(q)} 
 $$
 from the principal character, we obtain
\begin{equation}
\label{eq:E R}
E(\cA, \cB) \le \frac{A^2B^2}{\varphi(q)} + R, 
\end{equation}
where  
$$
R = =   \frac{1}{\varphi(q)}\sum_{\chi\in \cX_q*} 
\sum_{\substack{a_1, a_2 \in \cA\\
 \gcd(a_1a_2,q)=1}}  \ \sum_{\substack{b_1, b_2 \in \cB\\
 \gcd(b_1b_2,q)=1}}      \chi\(a_1a_2b_1^{-1}b_2^{-1}\).
$$
Rearranging, we obtain
$$
R =   \frac{1}{\varphi(q)}\sum_{\chi\in \cX_q*} 
\(\sum_{a\in \cA}   \chi\(a\)\)^2  \(\sum_{b \in \cB}      \overline\chi\(b\)\)^2, 
$$
where $\overline\chi$ is the complex conjugate character (note the co-primality 
conditions $\gcd(a,q)= \gcd(b,q)=1$ are abandoned from the last sum as redundant). 
Now, using the Cauchy inequality and recalling Lemma~\ref{lem:4th-Moment}, we conclude that 
$$
|R| \le  AB q^{o(1)}.
$$
Substituting this in~\eqref{eq:E R},   and  recalling the well-known lower bound
$$
\varphi(q) \gg  \frac{q}{\log \log(q+2)}
$$\
see~\cite[Theorem~328]{HardyWright}, we conclude the proof. 
\end{proof} 

  \subsection{Number of solutions to some congruences with reciprocals} 
 \label{sec:CongEqReipr}  
 An important tool in our argument is an upper bound on the number 
of solutions   $J_r(q;K)$ to the  congruence   
$$\frac{1}{x_1}+ \ldots+ \frac{1}{x_r}\equiv \frac{1}{x_{r+1}}+ \ldots+\frac{1}{x_{2r}}\bmod q, \quad 
1 \le x_1,  \ldots, x_{2r} \le K, 
$$
where $r =1, 2, \ldots$. 

We start with the trivial bound $1 \le K \le q$ we have
 \begin{equation}
\label{eq:J triv}
J_1(q;K) \ll K. 
\end{equation}

For $r \ge 2$ and arbitrary $q$ and $K$, good upper bounds on  $J_r(q;K)$  
are known only for $r =2$ and are due to 
Heath-Brown~\cite[Page~368]{H-B1}
(see the bound on the sums of quantities $m(s)^2$ in the 
notation of~\cite{H-B1}).  More precisely, we have:

\begin{lemma}
\label{lem:Cayley} For $1 \le K \le q$ we have 
$$
J_2(q;K) \le \(K^{7/2} q^{-1/2} + K^2\)q^{o(1)}.
$$
\end{lemma}

It is also shown by Fouvry and  Shparlinski~\cite[Lemma~2.3]{FouShp} that the bound of 
Lemma~\ref{lem:Cayley} 
can be improved on average over $q$ in a dyadic interval $[Q,2Q]$. The same argument 
also works for $J_r(q;K)$ without any changes. 

Indeed,  let  $J_r(K)$ be the number 
of solutions   to  the  equation 
$$
\frac{1}{x_1}+ \ldots+ \frac{1}{x_r} = \frac{1}{x_{r+1}}+ \ldots+\frac{1}{x_{2r}}, \qquad 
1 \le x_1,  \ldots, x_{2r} \le K, 
$$
where $r =1, 2, \ldots$. We recall  that by the result of Karatsuba~\cite{Kar} (presented in 
the proof of~\cite[Theorem~1]{Kar}), see also~\cite[Lemma~4]{BouGar1}, we have:

\begin{lemma}
\label{lem:Recipr Eq} For any fixed positive integer $r$,  we have 
$$
J_r(K) \le K^{r + o(1)}.
$$
\end{lemma}

Now repeating the argument of the proof of~\cite[Lemma~2.3]{FouShp} and using 
Lemma~\ref{lem:Recipr Eq}  in the appropriate place,  we obtain:

\begin{lemma}
\label{lem:Cayley Aver} For any   fixed positive integer $r$ and sufficiently large integers $1 \le K \le Q$, 
we have 
$$
\frac{1}{Q}\sum_{Q \le q \le 2Q} J_{r}(q;K) \le \(K^{2r} Q^{-1} + K^r\)Q^{o(1)}.
$$
\end{lemma}

\section{Proofs of Main Results}

\subsection{Proof of Theorem~\ref{thm:SALMNq}}


For an integer $u$ we define
$$
\langle u\rangle_q = \min_{k \in \Z} |u - kq|
$$
as the distance to the closest integer, 
 which is a multiple of $q$.

Changing the order of summation and then changing the variables
$$
x \mapsto \xbar \mand y \mapsto \ybar,
$$
 we obtain
 \begin{align*}
 \SALMNq &  =  \sum_{\ell \in \cL} \alpha_\ell  \sum_{x\in \Z_q^*}  \sum_{m\in \cM} \sum_{n \in \cN}
  \eq\(\ell \xbar \ybar +m x + ny \)\\
& =   \sum_{\ell \in \cL} \alpha_\ell \sum_{x,y\in \Z_q^*}   \eq\(\ell \xbar \ybar\)      \sum_{m\in \cM}  \eq(mx) \sum_{n \in \cN} \eq(n y).
\end{align*}
Hence
$$
 \SALMNq   =   \sum_{\ell \in \cL} \alpha_\ell 
\sum_{x,y\in \Z_q^*}  \mu_x \nu_y  \eq\(\ell \xbar \ybar\), 
$$
 with some  complex coefficients $\mu_x$ and $\nu_y$, satisfying
 \begin{equation}
 \label{eq:bound mu nu}
|\mu_x| \le  \min\left\{M,  \frac{q}{\langle x \rangle_q}\right \} \mand
|\nu_y| \le  \min\left\{N,  \frac{q}{\langle y \rangle_q}\right \}, 
\end{equation}
see~\cite[Bound~(8.6)]{IwKow}.

We now set  $I = \rf{\log (M/2)}$ and define $2(I+1)$ the sets
 \begin{equation}
 \label{eq:Set Q}
\begin{split}
\cQ_0^{\pm} &   =\{x \in \Z~: ~ 0 < \pm x \le q/M,\ \gcd(x,q)=1\}, \\
\cQ_i^{\pm}  & = \{x \in \Z~: ~\min\{q/2, e^{i} q/M\}\ge \pm x > e^{i-1}q/M,\\
& \qquad  \qquad  \qquad  \qquad  \qquad  \qquad  \qquad  \qquad   \qquad \gcd(x,q)=1 \}, 
\end{split}
\end{equation}
where $i = 1, \ldots, I$.
Similarly, we set  $J = \rf{\log (N/2)}$ and define $2(J+1)$ the sets
 \begin{equation}
 \label{eq:Set R}
\begin{split}
\cR_0^{\pm} &   =\{x \in \Z~: ~ 0 < \pm y \le q/N,\ \gcd(x,q)=1\}, \\
\cR_j^{\pm}  & = \{y \in \Z~: ~\min\{q/2, e^{j} q/N\}\ge \pm y > e^{j-1}q/N,\\
& \qquad  \qquad  \qquad  \qquad  \qquad  \qquad  \qquad  \qquad   \qquad \gcd(x,q)=1 \}, 
\end{split}
\end{equation}
where $j = 1, \ldots, J$.

Therefore, 
 \begin{equation}
\label{eq:SAIJ S0i}
 \SALMNq  \ \ll   \sum_{i=0}^I \sum_{j=0}^J \(\left|S_{i,j}^+\right | + \left|S_{i,j}^-\right|\),
\end{equation}
where
$$
S_{i,j} ^{\pm} =  \sum_{\ell \in \cL}\sum_{x \in \cQ_i^\pm}  \sum_{y \in \cR_j^\pm} 
 \alpha_\ell     \mu_x \nu_y  \eq\(\ell \xbar \ybar\),  \qquad i = 0, \ldots, I, \ j = 0, \ldots J.
$$

For $\lambda \in \Z_q$ we denote 
$$
T_i^{\pm}(\lambda) = \dsum_{\substack{\ell \in \cL, \ x \in \cQ_i^\pm \\ \ell x \equiv \lambda \bmod q}}  \alpha_\ell     \mu_x 
$$
and note that $T_i^{\pm}(\lambda) =0 $ unless $\lambda \in \Z_q^*$. 
Therefore, 
$$
S_{i,j} ^{\pm} =  \sum_{\lambda \in \Z_q^*}  \sum_{y \in \cR_j^\pm} T_i^{\pm}(\lambda) 
\nu_y  \eq\(\lambda \ybar\),  \qquad i = 0, \ldots, I, \ j = 0, \ldots J.
$$
Let us fix some integers $i\in [0,I]$ and $j \in [0,J]$. 

We now  fix some integer $r \ge 1$.

Below we present the argument in a general form with an arbitrary integer $r \ge 1$. 
We then apply it with $r =1$ and $r=2$ since we use Lemma~\ref{lem:Cayley}. However 
in the proof of Theorem~\ref{thm:SALMNq Aver} we use it in full generality.

Writing 
$$
\left |S_{i,j} ^{\pm} \right|  \le   \sum_{\lambda \in \Z_q^*} \left| T_i^{\pm}(\lambda)\right|^{1 -1/r}
 \left|  T_i^{\pm}(\lambda)^2\right|^{1/2r}
\left |\sum_{y \in \cR_j^\pm} \nu_y   \eq(m \ybar)\right|, 
$$
by the H{\"o}lder inequality, for every choice of the 
sign `$+$' or `$-$',  we obtain
 \begin{equation}
 \label{eq:Holder S}
\begin{split}
\left |S_{i,j} ^{\pm} \right|  & \le  
\( \sum_{\lambda \in \Z_q^*} \left| T_i^{\pm}(\lambda)\right|\)^{1 -1/r}
\( \sum_{\lambda \in \Z_q^*} \left| T_i^{\pm}(\lambda)\right|^2\)^{1/2r} \\
& \qquad \qquad \qquad \qquad \qquad \(  \sum_{\lambda \in  \Z_q}
\left |\sum_{y \in \cR_j^\pm} \nu_y   \eq(m \ybar)\right|^{2r} \)^{1/2r}.
\end{split}
\end{equation}

We observe that by~\eqref{eq:bound mu nu} and~\eqref{eq:Set Q}, for $x \in \cQ_i^\pm$ we have 
$$
\mu_x \ll e^{-i} M. 
$$
Hence 
 \begin{equation}
 \label{eq:T1}
\begin{split}
\sum_{\lambda \in \Z_q^*} \left| T_i^{\pm}(\lambda)\right| & \ll 
\sum_{\lambda \in \Z_q^*}  
\dsum_{\substack{\ell \in \cL, \ x \in \cQ_i^\pm \\ \ell x \equiv \lambda \bmod q}} \left| \alpha_\ell     \mu_x \right| \ll 
 e^{-i} M \sum_{\lambda \in \Z_q^*}  
 \dsum_{\substack{\ell \in \cL, \ x \in \cQ_i^\pm \\ \ell x \equiv \lambda \bmod q}} 1 \\
  & \ll 
 e^{-i} M  \# \cL \# \cQ_i^\pm \ll   e^{-i} L M  \(e^{i} q/M\) = qL.
\end{split}
\end{equation}
Similarly,
 \begin{align*}
\sum_{\lambda \in \Z_q^*} \left| T_i^{\pm}(\lambda)\right|^2
 & \ll 
 e^{-2i} M^2 \sum_{\lambda \in \Z_q^*}  
\( \dsum_{\substack{\ell \in \cL, \ x \in \cQ_i^\pm \\ \ell x \equiv \lambda \bmod q}} 1\)^2 
=  e^{-2i} M^2   E( \cL, \cQ_i^\pm), 
\end{align*}
where $E(\cA, \cB)$ is as defined in Lemma~\ref{lem:Energy}, which implies 
 \begin{equation}
 \begin{split}
 \label{eq:T2}
\sum_{\lambda \in \Z_q^*} \left| T_i^{\pm}(\lambda)\right|^2 & \le  e^{-2i} M^2   
 \(\frac{L^2  \(e^{i} q/M\) ^2}{q} + L   \(e^{i} q/M\) \) q^{o(1)}\\
 & \le q^{1+o(1)} L^2  + e^{-i} q^{1+o(1)} LM.
\end{split}
\end{equation}

Next, opening up the inner exponential sum in~\eqref{eq:Holder S}, changing the order of summation and 
using the orthogonality of exponential functions, we obtain 
 \begin{align*}
\sum_{\lambda \in  \Z_q}&\left |\sum_{y \in \cR_j^\pm} \nu_y   \eq(\lambda \ybar)\right|^{4} \\
 &   \le\sum_{\lambda \in  \Z_q}  \ssum_{y_1, \ldots y_{2r}  \in \cR_j^\pm }
\prod_{j=1}^r  \nu_{y_j}  \overline{\nu_{y_{r + j}}}
 \eq\(\lambda \sum_{j=1}^r \(\ybar_j - \ybar_{r + j}\)\)\\
 &   \le  \ssum_{y_1,  \ldots,  y_{2r} \in\cR_j^\pm}
\prod_{j=1}^r  \nu_{y_j}  \overline{\nu_{y_{r + j}}}
\sum_{\lambda \in  \Z_q}\eq\(m \sum_{j=1}^r \(\ybar_j - \ybar_{r + j}\)\)\\
 & =  q  \ssum_{\substack{y_1,  \ldots,  y_{2r} \in\cR_j^\pm\\ 
 \ybar_1 + \ldots + \ybar_r \equiv \ybar_{r+1}+ \ldots + \ybar_{2r} \bmod q}}
\prod_{j=1}^r  \nu_{y_j}  \overline{\nu_{y_{r + j}}}.
\end{align*}

We observe that by~\eqref{eq:bound mu nu} and~\eqref{eq:Set R} for $y \in \cR_j^\pm$ we have 
$$
\nu_y \ll e^{-j} N. 
$$
Hence,  
\begin{equation}
\label{eq:S 2r}
\begin{split}
\sum_{\lambda \in  \Z_q}\left |\sum_{y \in \cR_j^\pm} \nu_y   \eq(\lambda \ybar)\right|^{2r}  &
 \ll   e^{-2r j}  q N^{2r}    \ssum_{\substack{y_1,  \ldots,  y_{2r} \in \cR_j^\pm\\ 
 \ybar_1 + \ldots + \ybar_r \equiv \ybar_{r+1}+ \ldots + \ybar_{2r} \bmod q}}1\\
&  \le    e^{-2r j}  q N^{2r}   J_r(q;\fl{e^{j}q/N}).
 \end{split}
\end{equation}

Substituting~\eqref{eq:T1},  \eqref{eq:T2} and~\eqref{eq:S 2r}  in~\eqref{eq:Holder S}, we see that
 \begin{equation}
 \label{eq:Holder STT}
\begin{split}
\left |S_{i,j} ^{\pm} \right|  & \le  
\( qL\)^{1 -1/r}
\(q^{1+o(1)} L^2  + e^{-i} q^{1+o(1)} LM\)^{1/2r} \\
& \qquad \qquad \qquad \qquad \qquad \(  e^{-2r j}  q N^{2r}   J_r\(q;\fl{e^{j}q/N}\)\)^{1/2r}\\
& \le  e^{- j} LN
q^{1+o(1)}  \(1  + e^{-i}   M/L\)^{1/2r}     J_r\(q;\fl{e^{j}q/N}\)^{1/2r}.
\end{split}
\end{equation}


 Now using~\eqref{eq:Holder STT} with $r = 1$ and recalling~\eqref{eq:J triv}, we derive
 \begin{equation}
 \label{eq:Sij}
\begin{split}
\left |S_{i,j} ^{\pm} \right|  & \le e^{- j} LN
q^{1+o(1)}  \(1  + e^{-i/2}   (M/L)^{1/2}\)  e^{j/2}  N^{-1/2} q^{1/2}  \\
& \le    e^{-j/2}
  \(L + e^{-i/2}  L^{1/2} M^{1/2}\)  N^{1/2} q^{3/2+o(1)}. 
\end{split}
\end{equation}
Summing over all admissible $i$ and $j$ yields
\begin{equation}
\label{eq:Bound1}
\SALMNq  \le \(L +    L^{1/2} M^{1/2}\)  N^{1/2} q^{3/2+o(1)}.
\end{equation}

 Next, using~\eqref{eq:Holder STT} with $r = 2$ and invoking Lemma~\ref{lem:Cayley}, we obtain 
 \begin{align*}
\left |S_{i,j} ^{\pm} \right|  & \le e^{- j} LN
q^{1+o(1)}  \(1  + e^{-i/4}   (M/L)^{1/4}\) \\
& \qquad \qquad \qquad \qquad \quad \(e^{7j/8} N^{-7/8} q^{3/4}  + e^{j/2}  N^{-1/2}q^{1/2} \)  \\
& \le    
  \(L + e^{-i/4}  L^{3/4} M^{1/4}\) \(e^{-j/8}N^{1/8}  q^{7/4}  + e^{-j/2}  N^{1/2} q^{3/2}\) q^{o(1)}, 
\end{align*}
and we now derive
\begin{equation}
\label{eq:Bound2}
\SALMNq  \le 
 \(L +    L^{3/4} M^{1/4}\) \( N^{1/8}  q^{7/4}  +   N^{1/2} q^{3/2}\) q^{o(1)}. 
\end{equation}

Combining the bounds~\eqref{eq:Bound1} and~\eqref{eq:Bound2}, we obtain the result.

\subsection{Proof of Theorem~\ref{thm:SALMNq Aver}}

We proceed as in the proof of Theorem~\ref{thm:SALMNq}, in particular, 
we set $J = \rf{\log (N/2)}$. We also define  $K_j = \fl{2 e^{i}Q/N}$ 
and replace  $J_r(q;\fl{e^{j}q/N})$ 
with $J_r(q;K_j)$ in~\eqref{eq:Holder STT},  $j= 0, \ldots, J$.

We now see that by Lemma~\ref{lem:Cayley Aver} for every $j=0, \ldots, J$ for all
but at most $Q^{1-2r \varepsilon + o(1)}$ integers  $q \in [Q,2Q]$ we have 
\begin{equation}
\label{eq: Bound JqKi}
J_{r}(q;K_j) \le \(K_j^{2r} q^{-1} + K_j^r\)Q^{2 r \varepsilon}.
\end{equation}
Since $J = Q^{o(1)}$, for all
but at most $Q^{1-2r \varepsilon + o(1)}$ integers  $q \in [Q,2Q]$, 
the bound~\eqref{eq: Bound JqKi} holds for all $j=0, \ldots, J$ 
simultaneously.  

Now, for every such intreger $q$,  using~\eqref{eq: Bound JqKi} instead of the 
bound of  Lemma~\ref{lem:Cayley},  we obtain 
 \begin{align*}
\left |S_{i,j} ^{\pm} \right|  & \le  e^{- j} LN
q^{1+o(1)}  \(1  + e^{-i}   M/L\)^{1/2r}  \\
& \qquad \qquad \qquad \qquad \quad \(e^{j} N^{-1} q^{1-1/2r}  + e^{j/2}  N^{-1/2}q^{1/2} \)  \\
& \le    
  \(L + e^{-i/2r}  L^{1-1/2r} M^{1/2r}\) \(q^{2-1/2r}  + e^{-j/2}  N^{1/2} q^{3/2}\) q^{o(1)}, 
\end{align*}
instead of~\eqref{eq:Sij} for every $i=0, \ldots, I$ and $j = 0, \ldots, J$.  
Since $I, J= Q^{o(1)}$, the result now follows.




\end{document}